\documentclass[11pt]{article}

\usepackage{tikz}
\tikzstyle{vertex}=[circle, draw, inner sep=0pt, minimum size=5pt]

\usetikzlibrary{decorations.markings}
\usepackage{scalefnt}

\usepackage[english]{babel}
\usepackage{times}
\usepackage{amsthm}
\usepackage{amssymb, verbatim}
\usepackage{amsmath}
\usepackage{eucal}
\usepackage{mathrsfs}

\usepackage{mathabx}

\newtheorem{theorem}{Theorem}[section]
\newtheorem{corollary}[theorem]{Corollary}
\newtheorem{lemma}[theorem]{Lemma}
\newtheorem{proposition}[theorem]{Proposition}

\theoremstyle{definition}

\newtheorem{remark}[theorem]{Remark}
\newtheorem{da-scrivere}[theorem]{da-scrivere}

\newcommand{\Ad}{{\mathrm {Ad}}}
\newcommand{\Aut}{{\mathrm {Aut}}}
\newcommand{\out}{{\mathrm {Out}}}

\newcommand{\End}{\operatorname{End}}

\def\D{{\cal D}}

\def\F{{\cal F}}

\def\O{{\cal O}}

\def\Q{{\cal Q}}

\def\U{{\cal U}}

\begin{document}
\title{Diagonal automorphisms of the $2$-adic ring $C^*$-algebra}
\author{Valeriano Aiello$^\dag$, Roberto Conti$^\sharp$, Stefano Rossi$^\natural$
\\
\\
$^\dag$ Department of Mathematics\\ Vanderbilt University \\1362 Stevenson Center, Nashville, TN 37240, USA\\
valerianoaiello@gmail.com \\
\\
$^\sharp$ Dipartimento di Scienze di Base e Applicate per l'Ingegneria \\ Sapienza Universit\`a di Roma \\ Via A. Scarpa 16,
I-00161 Roma, Italy\\
roberto.conti@sbai.uniroma1.it\\
\\
$^\natural$ Dipartimento di Matematica,\\ Universit\`a di Roma Tor
Vergata\\ Via della Ricerca Scientifica 1, I--00133 Roma, Italy\\
rossis@mat.uniroma2.it}
\date{}
\maketitle

\begin{abstract}
The $2$-adic ring $C^*$-algebra $\Q_2$ naturally contains a copy of the Cuntz algebra $\O_2$ and, a fortiori, also of its diagonal subalgebra $\D_2$ with Cantor spectrum. This paper is aimed at studying 
the group $\Aut_{\D_2}(\Q_2)$ of the automorphisms of $\Q_2$ fixing $\D_2$  pointwise. It turns out that any such automorphism leaves $\O_2$ globally invariant. 
Furthermore, the subgroup $\Aut_{\D_2}(\Q_2)$ is shown to be  maximal abelian in $\Aut(\Q_2)$.
Saying exactly what the group is  amounts to understanding when an automorphism of $\O_2$ that fixes $\D_2$ pointwise extends to $\Q_2$.
A complete answer is given for all localized automorphisms: these will extend  if and only if they are the composition of a localized inner automorphism with a 
gauge automorphism.

\end{abstract}

\section{Introduction}
As soon as the Cuntz algebras were introduced in \cite{Cuntz1}, it was quickly realized that studying their endomorphisms and automorphisms  would initiate a fruitful research season.
The forecast could not possibly be more accurate, for after nearly forty years they continue to be a major topic and a source of inspiration, as  demonstrated by the recent literature that has been accumulating at an impressive rate,
see e.g. \cite{ContiSzymanski,CKS-MoreLoc,Conti,CRS,CHSCrelle,CHS,CHSJFA,CHS2015}.  
Motivated by these works,
we found it natural to ask ourselves whether such a case study would also provide the right tools to analyze the endomorphisms and automorphisms of other classes of $C^*$-algebras, 
notably those recently associated with rings, fields and other algebraic objects.
In fact, this was one of the main reasons why in \cite{ACR} we started an investigation of the group $\Aut(\Q_2)$ of unital $*$-preserving automorphisms of the so-called dyadic ring $C^*$-algebra of the integers $\Q_2$, a known $C^*$-algebra (see e.g. \cite{LarsenLi} and the references therein) associated to the semidirect product semigroup 
$${\mathbb Z} \rtimes \{1,2,2^2,2^3,\ldots\}$$ 
that contains a copy of the Cuntz algebra $\O_2$ in a canonical way. Inter alia, it was proved in \cite{ACR}  the useful fact that the canonical diagonal $\D_2$ maintains the property of being a maximal abelian subalgebra (MASA) in $\Q_2$ also. Actually, more is known, for $\D_2$ is even a Cartan subalgebra of both $\O_2$ and $\Q_2$ \cite[Prop. 4.3]{HaOl}. 
The study there initiated is further developed in the present paper. In particular, our main focus is here on the structure of the set of those automorphisms of $\Q_2$ leaving the diagonal MASA $\D_2$ pointwise fixed, which will always be denoted by  $\Aut_{\D_2}(\Q_2)$. 
Rather interestingly, this group turns out to be a maximal abelian subgroup of $\Aut(\Q_2)$. Moreover, we show that any of its elements restricts to an automorphism of $\O_2$ and it is indeed the unique extension of its restriction. 
It immediately follows from  the analysis carried out in \cite{Cuntz} that such restrictions are automorphisms of $\O_2$ induced by unitaries in $\D_2$, henceforth referred to as diagonal automorphisms for short. The results here obtained lend further support to the idea, already expressed in \cite{ACR}, that the group of automorphisms of $\Q_2$ is, in a sense, considerably smaller than that of $\O_2$, thus making it reasonable to ask the challenging question whether this group may be computed explicitly up to inner automorphisms.
Indeed, we show that any extendible localized diagonal automorphism of $\O_2$ is necessarily 
the product of a gauge automorphism and a localized inner diagonal automorphism.
The general case is still out of the reach of the techniques used in this paper instead. Even so, we do spot a necessary and sufficient condition for a diagonal automorphism to extend. 
Despite all our efforts to exploit the condition, to date this has not aided us in deciding whether the only extendible diagonal automorphisms of $\O_2$ are products of gauge and inner diagonal automorphisms.
It is quite possible that the answer to this problem will also require to delve further into the fine ergodic properties of the odometer map.

\section{Preliminaries and notations}
We  recall the basic definitions and properties of the $2$-adic ring $C^*$-algebra  so as to make the reading of the present paper suitable for a broader, not necessarily specialized, audience. 
This is the universal $C^*$-algebra $\Q_2$ generated by an isometry $S_2$ and a unitary $U$ such that $S_2S_2^*+US_2S_2^*U^*=1$ and $S_2U=U^2S_2$.
A very informative account of its most relevant properties is given in \cite{LarsenLi} as well as in our former work \cite{ACR}. As far as the purposes of this work are concerned, it is important to
mention that the Cuntz algebra $\O_2$, i.e. the universal $C^*$-algebra generated by two isometries $X_1$ and $X_2$ such that $X_1X_1^*+X_2X_2^*=1$, can be thought of as a subalgebra of $\Q_2$
via the injective unital $*$-homomorphism that sends $X_1$ to $US_2$ and $X_2$ to $S_2$. As of now this $*$-homomorphism will always be understood without explicit mention, therefore we  simply
write $\O_2\subset\Q_2$ to refer to the copy of $\O_2$ embedded in $\Q_2$ in this way.  The $2$-adic ring $C^*$-algebra is actually a kind of a more symmetric version of $\O_2$, in which the Cuntz isometries
$S_1$ and $S_2$ are now intertwined by the unitary $U$, to wit $S_2U=US_1$. This circumstance introduces a higher degree of rigidity, which is ultimately responsible for a shortage of outer automorphisms. 
Although not completely computed, the group $\out(\Q_2)$ is nevertheless closely looked over in \cite{ACR}, where it is shown to be considerably smaller than $\out(\O_2)$, while being still uncountable and noncommutative. 
To get a better idea of to what extent the former group is smaller than the latter, it is worthwhile to point up that endomorphisms or automorphisms of $\O_2$ will not in general extend to $\Q_2$, as widely discussed in \cite{ACR}.
Among those that do extend there are the canonical endomorphism, the gauge automorphisms and the flip-flop. Of these only the first two will play an important role in this work. In particular, the gauge automorphisms will
actually play an overriding role, which means they need a bit more exhaustive introduction.\\
On $\O_2$ these are the automorphisms $\alpha_\theta$  acting on each isometry simply multiplying it by
$e^{i\theta}$, i.e. $\alpha_\theta (S_i)=e^{i\theta}S_i$  for $i=1,2$, where $\theta$ is any real number.   The action of the one-dimensional torus $\mathbb{T}$ on $\O_2$ provided by 
the  gauge automorphisms enables us to speak of the gauge invariant subalgebra $\F_2\subset \O_2$, which is by
definition the $C^*$-subalgebra  whose elements are fixed by all the ${\alpha_\theta}$'s.  It is well known that $\F_2$ is the UHF algebra of type $2^\infty$.
Now the gauge automorphisms are immediately seen to extend to automorphisms $\widetilde{\alpha_\theta}$ of the whole $\Q_2$ with $\widetilde{\alpha_\theta}(U)=U$. Less obviously, each $\widetilde{\alpha_\theta}$
is an outer automorphism when it is not trivial, which is proved in \cite{ACR}. Furthermore, the extended gauge automorphisms allow us to consider the gauge invariant subalgebra of $\Q_2$, which we denote by $\Q_2^\mathbb{T}$. 
Among other things, $\Q_2^\mathbb{T}$ is known to be a Bunce-Deddens algebra. 
It is not particularly hard to prove that $\Q_2^{\mathbb{T}}$ can also be described 
as the $C^*$-subalgebra of $\Q_2$ generated from either $\F_2$ or $\D_2$ and $U$, where $\D_2\subset\F_2$ is a notable commutative subalgebra of the Cuntz algebra $\O_2$.
Commonly referred to as the diagonal subalgebra, $\D_2$ is in fact the subalgebra generated by the diagonal projections $P_\alpha\doteq S_\alpha S_\alpha^*$, where for any multi-index $\alpha=(\alpha_1,\alpha_2,\ldots,\alpha_k)\in\bigcup_n\{1,2\}^n$
the isometry $S_\alpha$ is the product $S_{\alpha_1}S_{\alpha_2}\ldots S_{\alpha_k}$. The multi-index notation is rather convenient when making computations in Cuntz algebras and will be used extensively throughout the paper.
In particular, we need to recall that $|\alpha|$ is the length of the multi-index $\alpha$.
By definition, the diagonal $\D_2$ is also the inductive limit of the increasing sequence of the finite-dimensional subalgebras $\D_2^k\subset\D_2^{k+1}\subset\D_2$ given by $\D_2^k\doteq\textrm{span}\{P_\alpha: |\alpha|=k\}$,
$k\in\mathbb{N}$.  
That $\D_2$ is quite a remarkable subalgebra is then seen at the level of its spectrum, for the latter is the Cantor set $K$, of which many a concrete topological realization is known. However, in what follows we shall always
think of it as the Tychonov infinite product $\{1,2\}^\mathbb{N}$.\\
Not of less importance is the canonical endomorphism $\varphi\in\End(\O_2)$, which is defined on each element $x\in\O_2$ as $\varphi(x)=S_1xS_1^*+S_2xS_2^*$. 
By its very definition it clearly extends to $\Q_2$, on which 
it still acts as a strongly ergodic map, namely $\bigcap_n\varphi^n(\Q_2)=\mathbb{C}1$, as shown in \cite{ACR}. The intertwining rules $S_i x = \varphi(x)S_i$ for any  $x\in {\mathcal Q}_2$ with $i=1,2$ still hold true.
 In addition, the canonical endomorphism preserves the diagonal $\D_2$, acting on its spectrum as the usual shift map on $\{1,2\}^\mathbb{N}$.\\
To complete the description of our framework we still need to single out a distinguished representation of $\Q_2$ among the many, namely the so-called canonical representation \cite{LarsenLi}. This is the representation in which $S_2$ and $U$
are concrete operators acting on the Hilbert space $\ell_2(\mathbb{Z})$ as $S_2e_k=e_{2k}$ and $Ue_k=e_{k+1}$, for every $k\in\mathbb{Z}$, where $\{e_k:k\in\mathbb{Z}\}$ is the canonical basis of 
$\ell_2(\mathbb{Z})$. Finally, in this representation $\D_2$ can be seen as a norm-closed subalgebra of $\ell_\infty(\mathbb{Z})$, the diagonal operators with respect to the canonical basis. For any $d\in\ell_\infty(\mathbb{Z})$ we
denote by $d(k)$ its $k$th diagonal entries, that is $d(k)\doteq (e_k,de_k)$, $k\in\mathbb{Z}$. It is obvious that $\textrm{Ad}(U)$ leaves $\ell_\infty(\mathbb{Z})$ invariant. It is slightly less obvious that it also leaves $\D_2$ globally invariant.
Moreover, the spectrum of $\D_2$ is acted on by $\textrm{Ad}(U)$ through the homeomorphism given by the so-called odometer on the Cantor set, which is known to be a uniquely ergodic map, see \cite{Davidson}.
Finally, the action of $\textrm{Ad}(U)$ on $\D_2$ is compatible with its inductive-limit structure, i.e. $\textrm{Ad}(U)(\D_2^k)=\D_2^k$ for every $k\in\mathbb{N}$.



\section{General structure results on $\Aut_{\D_2}(\Q_2)$} 
 
Given an inclusion of $C^*$-algebras $A\subseteq B$, we shall denote by $\Aut(B, A)$ the group of those automorphisms of $B$ leaving $A$ globally invariant, and  by $\Aut_A(B)$ the group of those fixing $A$ pointwise. Both endomorphisms and automorphism 
are tacitly assumed to be unital whenever our $C^*$-algebras are unital, as $\D_2$ and $\Q_2$ are.  
Having a general content, the following result may possibly be known, cf. \cite{Cuntz} for instance. Because we have no explicit reference, a precise statement is nevertheless included along with its proof. 
To state it as clearly as possible, however, we still need to set some notations.  In particular, we recall that if $H$ is any subgroup of a group $G$, its normalizer is the largest subgroup $N_H(G)$  in which $H$ is contained as a normal subgroup. More explicitly, $N_H(G)$ can be identified with the set of those $g\in G$ such that $gHg^{-1} = H$. Finally, if $S$ is any subset of $\Aut(B)$, we denote by $B^S$ the sub-$C^*$-algebra of $B$ whose elements are fixed by all automorphisms of $S$. Here follows the result.
\begin{proposition}\label{Galois}
Let $A \subseteq B$ be a unital inclusion of $C^*$-algebras. Then 
\begin{enumerate}
\item $\Aut(B,A) \subseteq N_{\Aut_A(B)}(\Aut(B))$;
\item If $B^{\Aut_A(B)} = A$ one has $\Aut(B,A) = N_{\Aut_A(B)}(\Aut(B))$;
\item If $A$ is a MASA in $B$, then $B^{\Aut_A(B)} = A$ and  $\Aut(B,A) = N_{\Aut_A(B)}(\Aut(B))$.
\end{enumerate}
\end{proposition}
\begin{proof}
For the first property, let $\gamma\in \Aut(B,A)$ and $\beta\in \Aut_A(B)$. Since $\gamma\circ\beta\circ\gamma^{-1}(a)=a$ for all $a\in A$, we get that $\gamma$ belongs to $N_{\Aut_A(B)}(\Aut(B))$.
\noindent
For the second, consider $\gamma\in  N_{\Aut_A(B)}(\Aut(B))$ and $\beta\in \Aut_A(B)$. We have that $\gamma\circ\beta\circ\gamma^{-1}(a)=a$ for all $a\in A$, hence $\beta\circ\gamma^{-1}(a)=\gamma^{-1}(a)$. This means that $\gamma^{-1}(a)\in B^{\Aut_A(B)}$ which is by hypothesis equal to $A$, thus $\gamma\in \Aut(B,A)$.
Finally, for the third we observe that any $u\in \U(A)$ gives rise to an element of $\Aut_A(B)$, namely $\textrm{Ad}(u)$. This means that for any $x\in B^{\Aut_A(B)}$ we have $uxu^*=x$. Since $A$ is a MASA, then $x\in A$ and the second property in turn implies that $\Aut(B,A) = N_{\Aut_A(B)}(\Aut(B))$.
\end{proof}
By applying the former result to the inclusion $\D_2\subset\Q_2$, the set equality $\Q_2^{\Aut_{\D_2}(\Q_2)} = \D_2$ is immediately got to. Now $\Aut_{\D_2}(\Q_2)$ contains both $\{\widetilde{\alpha_\theta} \ | \ \theta \in {\mathbb R}\}$ and 
$\{\textrm{Ad}(u),  \ u \in \U(\D_2)\}$, where $\widetilde{\alpha_\theta}$ is the unique extension to $\Q_2$ of the gauge automorphism $\alpha_\theta\in\Aut(\O_2)$, see \cite{ACR}. Furthermore, the intersection  $\Aut_{\D_2}(\Q_2) \bigcap {\rm Inn}(\Q_2)$ 
is easily seen to reduce to $\{\textrm{Ad}(u),  u \in \U(\D_2)\}$ thanks to maximality of $\D_2$ again. Since $\D_2$ is globally invariant under $\textrm{Ad}(U)$, for any $\alpha \in \Aut_{\D_2}(\Q_2)$ we have
$$\alpha(U d U^*) = U d U^* = \alpha(U) d \alpha(U)^*\,\textrm{for all}\, d \in \D_2 \ $$
that is $U^* \alpha(U)$ commutes with every $d \in \U(\D_2)$ and therefore by maximality $U^*\alpha(U)=\dj_\alpha$ for some $\dj_\alpha \in \U(\D_2)$, which we rewrite as
$$\alpha(U) = U \dj_\alpha$$ 
It is also clear that $$\alpha^{-1}(U) = U \dj_\alpha^*$$
To take but one example, when $\alpha$ is an inner automorphism, say $\textrm{Ad}(u)$ for some $u$ in $\U(\D_2)$, the corresponding  $\dj_{\textrm{Ad} (u)}$ is nothing but  $U^* u U u^*$. More importantly, the map $\alpha \mapsto \dj_\alpha$ is easily  recognized to be a group homomorphism between $\Aut_{\D_2}(\Q_2)$ and $\U(\D_2)$. Moreover, its kernel coincides with $\Aut_{\Q_2^{\mathbb T}}(\Q_2)$, cf. Proposition \ref{ker-check}. 
Our next goal is to show that the Cuntz-Takesaki unitary $u_\alpha\doteq\alpha(S_1)S_1^*+\alpha(S_2)S_2^*$ belongs to $\D_2$ as well.
\begin{proposition}
Let $\alpha$ be in $\Aut_{\D_2}(\Q_2)$. Then the corresponding unitary $u_\alpha$ lies in $\D_2$. 
\end{proposition}
\begin{proof}
First we observe that $\alpha(S_i)=u_\alpha S_i$  for $i=1,2$.
By maximality of $\D_2\subset\Q_2$, it is enough to prove that $u_\alpha$ commutes with the generating projections $P_{i_1i_2\ldots i_k}$ of $\D_2$. This can be easily seen by induction on $k$, as done by Cuntz for $\O_2$. The case of length one reduces to the computation $P_i=\alpha(P_i)=\alpha(S_iS_i^*)=u_\alpha S_iS_i^*u_\alpha^*=u_\alpha P_i u_\alpha^*$. The case of lenght two entails the computation $P_{ij}=\alpha(P_{ij})=\alpha(S_iS_jS_j^*S_i^*)=u_\alpha S_i (u_\alpha S_jS_j^*u_\alpha^*)S_i^*u_\alpha^*=u_\alpha S_i S_jS_j^*S_i^*u_\alpha^*=u_\alpha P_{ij} u_\alpha^*$. It is now clear how to go on.
\end{proof}
The following result can be derived at once from the foregoing proposition. In this respect, it is worth  recalling the fact, proved in \cite{Cuntz}, that there exists an explicit group isomorphism between
$\U(\D_2)$ and $\Aut_{\D_2}(\O_2)$ given by $d \mapsto \lambda_d$, where $\lambda_d(S_i) = d S_i$, $i=1,2$, already showing the abelianness of $\Aut_{\D_2}(\O_2)$. 
\begin{corollary}
Any $\alpha\in\Aut_{\D_2}(\Q_2)$ restricts to an automorphism of the Cuntz algebra $\O_2$ fixing pointwise the diagonal $\D_2$ and it is the unique extension
of such restriction. In particular, the group $\Aut_{\D_2}(\Q_2)$ is abelian.
\end{corollary}
\begin{proof}
It is clear from the previous proposition that $\alpha(S_i) \in \O_2$ for $i=1,2$, so that $\alpha(\O_2) \subseteq \O_2$. Now, $u_\alpha^* S_i \in \O_2$ and $\alpha(u_\alpha^* S_i) = u_\alpha^* \alpha(S_i) = u_\alpha^* u_\alpha S_i = S_i$, $i=1,2$, thus showing that indeed $\alpha(\O_2) = \O_2$. We conclude that $\alpha$ is an extension to $\Q_2$ of its restriction to $\O_2$,
and the statement about uniqueness follows at once from the rigidity result, proved in \cite[Sect. 4]{ACR}, that two automorphisms of $\Q_2$, coinciding on $\O_2$, must be the same.
For the last claim, if $\alpha_i \in \Aut_{\D_2}(\Q_2)$, $i=1,2$, we compute $\alpha_1(\alpha_2(U)) = \alpha_1(U \dj_{\alpha_2}) = U \dj_{\alpha_1} \dj_{\alpha_2} = U \dj_{\alpha_2} \dj_{\alpha_1}= \alpha_2(\alpha_1(U))$ and, similarly, $\alpha_1(\alpha_2(S_2)) = \alpha_1(u_{\alpha_2}(S_2)) = u_{\alpha_2} u_{\alpha_1} S_2= \alpha_2(\alpha_1(S_2))$. The conclusion readily follows.
\end{proof}
Every element in $\Aut_{\D_2}(\Q_2)$ can thus be written as the unique extension of an element $\lambda_d \in \Aut_{\D_2}(\O_2)$, for some $d \in \U(\D_2)$. We denote such extension as $\widetilde{\lambda_d}$. However, one should not expect  all automorphisms $\lambda_d$ with $d \in \U(\D_2)$ to extend to $\Q_2$.
Denoting by $\widetilde{\U}(\D_2)$ the set of all $d \in \U(\D_2)$ such that $\lambda_d$ is extendible to an automorphism of $\Q_2$, it is then easy to deduce from the above discussion that $\widetilde{\U}(\D_2)$ is a actually a group (a subgroup of $\U(\D_2)$) and there exists a group isomorphism between $\widetilde{\U}(\D_2)$ and $\Aut_{\D_2}(\Q_2)$ given by $d \mapsto \widetilde{\lambda_d}$. 
Thus far we have seen that any $\alpha \in \Aut_{\D_2}(\Q_2)$ acts on $U$ as $\alpha(U) = U \dj_\alpha$ for some $\dj_\alpha \in \D_2$.  Now it is also possible to rewrite this
relation in the form $\alpha(U) = \check{\dj}_\alpha U$, where $\check{\dj}_\alpha$ is simply given by $U \dj_\alpha U^*$ and is still a unitary of $\D_2$. As $\alpha = \widetilde{\lambda_d}$,  we can simply write $\check{d}$ instead of 
$\check{\dj}_\alpha = \check{\dj}_{\widetilde{\lambda_d}}$ for $d \in \widetilde{\U}(\D_2)$.  For the same reason as above, the map $d\mapsto\check{d}$ is a group homomorphism from $\widetilde{\U}(\D_2)$ to $\U(\D_2)$. In fact, this map will turn out to be vital in the next sections. Contrary to 
what one might expect, though, it has proved to be a difficult task to establish a priori whether it is norm continuous, possibly because determining its domain $\widetilde{\U}(\D_2)$
is just another way to recast our main problem. Nevertheless, its kernel can be described quite explicitly.
\begin{proposition}\label{ker-check}
The kernel of the map $d\mapsto\check{d}$ (defined on $\widetilde\U(\D_2)$) is the subgroup of the gauge automorphisms. Actually, one has 
$$\Aut_{\Q_2^{\mathbb{T}}}(\Q_2)=\Aut_{\F_2}(\Q_2)=\{\widetilde{\alpha_\theta}: \theta\in\mathbb{R}\} \ .$$
\end{proposition}
\begin{proof}
Clearly the condition $\check{d}=1$ is the same as $\widetilde{\lambda_d}(U)=U$. Since $\Q_2^{\mathbb{T}}$ coincides with $C^*(U,\D_2)$, see e.g. \cite[Sect. 2]{ACR}, it means that $\widetilde{\lambda_d}\in\Aut_{\Q_2^{\mathbb{T}}}(\Q_2)$ so that $\lambda_d \in \Aut_{\O_2^{\mathbb{T}}}(\O_2)$. The conclusion now readily follows from
the fact that the automorphisms of $\O_2$ fixing the canonical UHF subalgebra $\F_2 = \O_2^{\mathbb T}$ pointwise are precisely the gauge automorphisms, as proved by Cuntz in \cite{Cuntz}.
\end{proof}

In particular, the restriction map $\Aut_{\D_2}(\Q_2)\ni\lambda\rightarrow\lambda\upharpoonright_{\O_2}\in\Aut_{\D_2}(\O_2)$ induces a group embedding which allows us to think of the fomer group as
a subgroup of the latter. Therefore, as of now we will simply write $\Aut_{\D_2}(\Q_2)\subset\Aut_{\D_2}(\O_2)$ to mean that.  
Of course the inclusion is proper. In other words, not all the automorphisms of $\O_2$ that leave the diagonal $\D_2$ globally invariant will extend. 
As a matter of fact, very few automorphisms can be extended. Although we do not have a general explicit description of all extendible automorphisms yet, we do have a complete description for
a particular class  of  automorphisms. This is just the subgroup $\Aut_{\D_2}(\O_2)_{\textrm{loc}}$ of those localized automorphisms we mentioned above in passing. Actually, the terminology comes  from Quantum Field Theory. Roughly speaking, 
an automorphism is localized when it preserves the union of the matrix subalgebras. More precisely, an automorphism $\lambda_u\in\Aut(\O_n)$ is said to be \emph{localized} when the corresponding unitary $u\in\U(\O_n)$
belongs to the algebraic dense subalgebra $\bigcup_k \F^n_k\subset\O_n$, where $\F^n_k$ is generated by the elements of the form $S_\alpha S_\beta^*$ with $\alpha , \beta\in\{1, \ldots , n\}^k$.
Furthermore, the inclusion $\Aut_{\D_2}(\Q_2)\subset\Aut_{\D_2}(\O_2)$ allows us to define a subgroup $\Aut_{\D_2}(\Q_2)_{\textrm{loc}}$ as the intersection $\Aut_{\D_2}(\Q_2)\bigcap\Aut(\O_2)_{\textrm{loc}}$.
As maintained in the abstract, we will prove that $\Aut_{\D_2}(\Q_2)_{\textrm{loc}}$ is so small that the sole localized automorphisms fixing $\D_2$ that extend are the composition of a localized inner automorphism with a gauge automorphism. \\

Before going on with our discussion, we would like to point out a remark for the sake of completeness.

\begin{remark}\label{extension-auto-D2}
Let $d\in \U(\D_2)$ and consider the associated automorphism $\lambda_d$ of $\O_2$. If $\lambda_d$ extends to an endomorphism $\lambda$ of $\Q_2$, then $\lambda$ is actually an automorphism, i.e. $d \in \widetilde\U(\D_2)$ and $\lambda=\widetilde{\lambda_d}$. Indeed, $\lambda(\Q_2)$ contains $\lambda_d(\O_2)=\O_2$. Moreover, $\lambda(U) = \tilde{d} U$ for a suitable $\tilde{d} \in \U(\D_2)$ (same argument as for automorphisms), so that
$\lambda(\tilde{d}^*U)=\tilde{d}^*\tilde{d}U=U$. All in all, the extension is nothing but $\widetilde{\lambda_d}$ (and $\tilde{d} = \check{d}$). 
\end{remark}

Going back to $\Aut_{\D_2}(\Q_2)$, we have shown it is abelian, but we want to improve our knowledge by proving it is also maximal abelian in $\Aut(\Q_2)$, in a way that closely resembles what happens for the Cuntz algebra $\O_2$ \cite{Cuntz}. 
Here follows the proof.
\begin{theorem}
The subgroup $\Aut_{\D_2}(\Q_2)$ is maximal abelian in $\Aut(\Q_2)$.
\end{theorem}
\begin{proof}
Let $\alpha$ be an automorphism of $\Q_2$ that commutes with $\Aut_{\D_2}(\Q_2)$. In particular $[\alpha,\textrm{Ad}(u)]=0$ for every $u\in\U(\D_2)$, to wit $\textrm{Ad}(\alpha(u))=\textrm{Ad}(u)$. As the center of $\Q_2$ is trivial, we see that $\alpha(u)=\chi(u)u$ for every $u\in\U(\D_2)$, where $\chi$ is a character of the group $\U(\D_2)$. Our result will be proved once we show $\chi(u)=1$ for every $u\in\U(\D_2)$. To this aim, note that the equality $\alpha(u)=\chi(u)u$ says that $\D_2$ is at least globally invariant under the action of $\alpha$.  With a slight abuse of notation, we still denote by $\alpha$ the restriction of $\alpha$ to $\D_2\cong C(K)$, where $K$ is the Cantor set. Let $\Phi\in\textrm{Homeo}(K)$ such that $\alpha(f)=f\circ\Phi$ for every $f\in C(K)$. The identity obtained above is then recast in terms of $\Phi$ as $f\circ\Phi=\chi(f)f$ for every $f\in C(K,\mathbb{T})$. We claim that $\chi(\U(\D_2))\subset\mathbb{T}$ is at most countable. If so, the theorem can now be easily inferred. Indeed, if $\Phi$ is not the identity map, then there exists $x\in K$ such that $\Phi(x)\neq x$. Then pick a function 
$f\in C(K,\mathbb{T})$ such that $f(x)=1$ and $f(\Phi(x))=e^{i\theta}$. The equality $f\circ\Phi=\chi(f)f$ evaluated at $x$ gives
$\chi(f)=e^{i\theta}$, that is $\chi$ is onto $\mathbb{T}$. To really achieve the result we are thus left with the task of proving the claim. This should be quite a standard fact from ergodic theory. However, we do give a complete proof. If $\mu$ is any Borel $\Phi$-invariant measure on $K$ , we can consider the Hilbert space $L^2(K,\mu)$, which is separable because $K$ is metrizable, and the Koopman unitary operator $U_\phi$ associated with $\Phi$, whose action is simply given by $U_\Phi(f)=f\circ\Phi$ a.e. for every $f\in L^2(K)$. As eigenfunctions of $U_\Phi$ associated with different eigenvalues are orthogonal and $\chi(f)$ is an eigenvalue
for every $f\in C(K,\mathbb{T})$, we see that $\{\chi(f):f\in C(K,\mathbb{T})\}$ is a countable set by virtue of separability.
\end{proof}

\begin{remark} 
We can also provide an alternative argument for the above result, proving more directly that $\alpha(P_\beta)=P_\beta$ for all the multi-indeces $\beta$. First of all we observe that the relation $\alpha(d)=\chi (d)d$ implies that the spectrum of the unitary $d$ is invariant under the rotation of $\chi (d)$. We begin with the case of $P_1$. Consider the unitary $d_1=P_1+e^{i 2\pi\theta}P_2$ with       
$\theta\neq \pm 1$. On the one hand, we know that $\alpha(d_1)=\chi(d_1)d_1$. On the other hand, since the spectrum of $d_1$ is not invariant under non-trivial rotations, we find that $\chi(d_1)$ must be $1$. The same reasoning applies to the unitary $\widetilde{d}_1=P_1-e^{i\theta}P_2$ too, and so we get the equality $\chi(\widetilde{d}_1)=1$, hence $\alpha(P_1)=\alpha\left(\frac{d_1+\widetilde{d}_1}{2}\right)=P_1$. We now deal with the general case of a $P_\beta$ with $\beta$ being a multi-index of length $k$ in much the same way. Consider the two unitary operators $d_\beta =\sum_{|\gamma |=k, \gamma\neq \beta} e^{i\theta} P_\gamma + P_\beta$ and $\tilde{d}_\beta =-\sum_{|\gamma |=k, \gamma\neq \beta} e^{i\theta} P_\gamma + P_\beta$. By the same 
argument as above we still find both $\alpha(d_\beta)=d_\beta$ and 
$\alpha(\widetilde d_\beta)=\widetilde d_\beta$,  and thus $\alpha(P_\beta)=\alpha\left(\frac{d_\beta+\tilde{d}_\beta}{2}\right)=P_\beta$ and we are done.
\end{remark}

\section{Necessary and sufficient conditions for extendability}
Thanks to the results achieved in the last section,  giving a complete non-tautological description of $\Aut_{\D_2}(\Q_2)$ entails studying
those unitaries $d\in\U(\D_2)$ for which the corresponding $\lambda_d\in\Aut(\O_2)$ may be extended to $\Q_2$. This section is mainly concerned with problems of this sort. 
When an automorphism $\lambda_d$ extends, we will say every so often  that the corresponding
$d$ is  extendible itself. This is undoubtedly a slight abuse of terminology, but it aids brevity.
Here follows our first result.
\begin{lemma}\label{unirel}
Let $d$ be in $\U(\D_2)$. Then $\lambda_d\in\Aut(\O_2)$ extends to an endomorphism of $\Q_2$ if and only if there exists a $\tilde{d}$ in $\U(\D_2)$ such that
\begin{align}
\tilde{d}UdS_1&=dS_2\tilde{d}U\label{Eq-ext-1}\\
\tilde{d}UdS_2&=dS_1\label{Eq-ext-2}
\end{align}
Moreover, such an extension is automatically an automorphism whenever it exists, i.e. $d \in \widetilde\U(\D_2)$, and $\tilde{d} = \check{d}$.  
\end{lemma}
\begin{proof}
If $\lambda_d$ extends, then the two equalities in the statement are easily verified with $\tilde{d} = \check{d}$ if one applies its extension $\widetilde{\lambda_d}$ (cf. Remark \ref{extension-auto-D2}) to $US_1=S_2U$ and $US_2=S_1$  respectively
also taking into account that $\widetilde{\lambda_d}(U)=\check{d}U$. The converse is dealt with analogously by noting that the pair ($\tilde{d}U, dS_2)$ in $\Q_2$ still satisfies the defining relations
of $\Q_2$ and therefore, by universality, there exists an endomorphism $\lambda$ of $\Q_2$ such that $\lambda(S_2) = dS_2$ and $\lambda(U) = \tilde{d}U$. But then,
$\lambda(S_1) = \lambda(US_2) = \tilde{d}UdS_2 = dS_1$ by Eq. (\ref{Eq-ext-2}), so that $\lambda$ extends $\lambda_d$. 
\end{proof}
At this point, the reader may be wondering whether it ever happens that $d=\check{d}$. In fact, it turns out that this is never the case unless $d=1$, namely we have the following result. 
\begin{proposition}\label{fixed-point-check}
The unitary $d=1$ is the unique fixed point of the map $\widetilde\U(\D_2)\ni d\mapsto\check{d}\in \U(\D_2)$.
\end{proposition}
\begin{proof}
If we work in the canonical representation, we simply need to show that $d(k)=1$ for every $k\in\mathbb{Z}$. We first handle the even entries.
Formula \eqref{Eq-ext-2} becomes $dUdS_2=dS_1$, which in turn gives 
\begin{align*}
dS_2 & =U^*S_1=S_2\; .
\end{align*} 
Now by computing the above equality on the vectors of the canonical basis of $\ell_2(\mathbb{Z})$ we get  $d(2k)=1$ for all $k\in \mathbb{Z}$.  As for the odd entries, Formula \eqref{Eq-ext-1} leads to $dUdS_1=dS_2dU$, which yields $d(2k+1)=d(k+1)$ for all $k\in\mathbb{Z}$. 
This in turn says all odd entries of $d$ are $1$ as well, apart from $d(1)$, which is in fact not determined by this condition. However, it cannot be different from $1$, for otherwise $d$ would not even belong to $\D_2$.
\end{proof}
Although more focused on the Cuntz algebra $\O_2$, the next useful result is included all the same. In fact, we do believe that it may shed some light on 
applications yet to come. Recall that if $d, d'\in\U(\D_2)$, then $\Ad(d')\circ \lambda_d=\lambda_{d'd\varphi(d')^*}$. In particular, taking $d'=d^*$ we get that $\lambda_d$ is extendible if and only if $\lambda_{\varphi(d)}$ is extendible.
\begin{proposition}\label{udustar}
Let $\lambda_d\in\Aut(\O_2)$ be an extendible automorphism. Then either $\lambda_d$ is a gauge automorphism or $\lambda_{dUd^*U^*}$ is outer.
\end{proposition}
\begin{proof}
First of all we prove that
\begin{align}\label{formula-check-phi-AAA}
	\widecheck{\varphi(d)}S_1=S_1\; .
\end{align}
To this aim, rewrite Formula \eqref{Eq-ext-2} as $\check{d}UdS_2 =  \check{d}UdU^*S_1   = d S_1$. Then, by using the identity $\widecheck{d\varphi(d)^*}=d Ud^*U^*$ and the multiplicativity of  the
map  $\,\check{}\; $ one finds
\begin{align}
& \check{d}S_1  = d Ud^*U^*S_1 = \widecheck{d\varphi(d)^*}S_1  = \check{d}\widecheck{\varphi(d)^*}S_1 \label{check-d-S1}
\end{align}
whence the claim. Likewise, Formula \eqref{Eq-ext-1} yields $\check{d}UdS_1  = \check{d}UdU^*S_2U  =d S_2 \check{d} U=d \varphi(\check{d})S_2  U$ and thus $\check{d}UdU^*S_2  = d \varphi(\check{d})S_2$, that is
\begin{align}
& \check{d} d^* UdU^*\varphi(\check{d}^*)S_2 = S_2\; . \label{check-d-S2}
\end{align}

The former equality actually shows that the automorphism $\Ad(\check{d}^*)\circ \lambda_{dUd^*U^*}$ fixes $S_2$, i.e. $\Ad(\check{d}^*)\circ \lambda_{dUd^*U^*}[S_2]=S_2$. By  \cite[Corollary B]{Matsumoto}, then either $\Ad(\check{d}^*)\circ \lambda_{dUd^*U^*}$ is the identity or $\Ad(\check{d}^*)\circ \lambda_{dUd^*U^*}$ is an outer automorphism of $\O_2$. 
In the first case $\check{d}\varphi(\check{d})^* = dUd^*U^*= \widecheck{d\varphi(d)^*}= 
\check{d}\widecheck{\varphi(d)}^*$, hence
\begin{align*}
& \varphi(\check{d}) = \widecheck{\varphi(d)} \ . 
\end{align*}
By applying the last equality to \eqref{formula-check-phi-AAA} we get 
\begin{align*}
& S_1=\widecheck{\varphi(d)}S_1=\varphi(\check{d})S_1=S_1\check{d}
\end{align*}
which proves that $\check{d}=1$, thus $\lambda_d$ is a gauge automorphism by Proposition \ref{ker-check}. In the second case, clearly $\lambda_{dUd^*U^*}$ is an outer automorphism of $\O_2$.
\end{proof}
At any rate, a first application can be given at once. 
\begin{corollary}
If $\lambda_d$, $d \in \U(\D_2)$ is a non-trivial inner automorphism of $\O_2$, then $\lambda_{UdU^*}$ is outer.
\end{corollary}
We can now resume to our general discussion. With this in mind, we start by spotting a useful necessary condition on $d$ for the corresponding $\lambda_d$ to extend. 
\begin{proposition}\label{pointspectrum}
Let $d$ be a unitary in $\D_2$. If $\lambda_d$ extends to $\Q_2$, then $d(0)=d(-1)$.
\end{proposition}
\begin{proof}
Owing to the extendability of $\lambda_d$ the isometries $dS_1=\lambda_d(S_1)$ and $dS_2=\lambda_d(S_2)$ are still intertwined, i.e. they are unitarily equivalent. In particular, their point spectra must coincide. 
Now the equalities $\sigma_p(dS_1)=\{d(-1)\}$ and $\sigma_p(dS_2)=\{d(0)\}$ are both easily checked, hence the thesis follows. 
\end{proof}
It goes without saying that the condition is only necessary. Even so, it does have the merit of highlighting a property of which we will have to make an extensive use. Therefore, unless otherwise stated, our unitaries $d\in\U(\D_2)$ will always satisfy 
the condition $$d(0)=d(-1) \ . $$ 
In addition, there 
is no lack of generality if we further assume that both $d(0)$ and $d(-1)$ equal $1$.
For if this were not the case, we could always multiply $\lambda_d$ by a suitable gauge automorphism, which of course would not affect the extendability of $\lambda_d$, since gauge automorphisms  certainly extend to $\Q_2$.
Finally, we will work in the canonical representation and
adopt the Dirac bra-ket notation for rank-one operators: for any given $u,v\in\ell_2(\mathbb{Z})$ the operator $w\rightarrow(v,w)u$ is  denoted by $|u \rangle \langle v |$.
This said, we can now state a result which sheds further light on the relation between $d$ and $\check{d}$ by relating our setup to some findings in \cite{LarsenLi}.

\begin{theorem}
Let $d$ be a unitary in $\D_2$ such that $d(0)=d(-1)=1$; then
\begin{enumerate}
\item there exists the strong limit $d_\infty$ of the sequence $d_k  \doteq d\varphi(d)\varphi^2(d)\ldots\varphi^{k-1}(d)$ in the canonical representation;
\item $\lambda_d$ is weakly inner in the canonical representation restricted to $\O_2$; 
\item $\lambda_d$ extends to a representation of $\Q_2$ on $\ell_2(\mathbb Z)$; 
\item the automorphism $\lambda_d\in\Aut(\O_2)$ extends to an automorphism of $\Q_2$
if and only if,
for a unique $\alpha\in \mathbb{T}$, the strong limit of the sequence $$x_k\doteq 
\alpha |e_0 \rangle \langle e_0 |+d\varphi(d)\varphi^2(d)\ldots\varphi^{k-1}(d)\left(\sum_{i=0}^{k-1}S_2^iS_1S_1^*(S_2^*)^i \right)U\varphi^{k-1}(d^*)\ldots\varphi(d^*)d^*U^*$$
belongs to $\D_2$, in which case 
the limit coincides with $\check{d}$ and $\alpha = \check{d}(0)$;
\item if $\lambda_d$ extends to an automorphism $\widetilde{\lambda_d}$, 
then $\widetilde{\lambda_d}$ is weakly inner in the canonical representation if ($\lambda_d$ is inner or)
$\check{d}(0)=1$.
\end{enumerate}
\end{theorem}

\begin{proof}
To begin with, we note that for every $j\in\mathbb{Z}$ the sequence $\{d_k(j): k\in\mathbb{N}\}\subset\mathbb{T}$ is eventually constant, as we certainly
have  $d_k(j)=d_{|j|}(j)$ for all $k\geq |j|$ thanks to $\varphi(d)(k)=d([k/2])$, $k\in\mathbb{Z}$, and $d(0)=d(-1)=1$. In particular, this shows that $\{d_k:k\in\mathbb{N}\}$ is
strongly convergent to a unitary $d_\infty\in\ell_\infty(\mathbb{Z})$.
Actually, the operator $d_\infty$ thus exhibited
implements $\lambda_d$, i.e.
$d S_i=d_\infty S_i d_\infty^*$, $i=1,2$.
Indeed, 
\begin{equation}\label{weaklyinn}
d_\infty S_i d_\infty^* = \lim_k d_k S_i d_k^* = \lim_k d_k \varphi(d_k)^* S_i \\
= \lim_k d \varphi^k(d^*) S_i = dS_i \ , 
\end{equation}
where we have used the equalities $S_i x = \varphi(x) S_i$, for all $x \in \O_2$ and $i=1,2$ and that $\varphi^k(d)$ strongly converges to 1.
Because $\Ad(d_\infty)$ restricts to $\O_2$ as $\lambda_d$, it also restricts to $\Q_2$ if only as a representation of the latter algebra on the Hilbert
space $\ell_2(\mathbb{Z})$.\\
To ease some of the computations we need to make, it is now particularly convenient to introduce 
the sequence of projections $Q_k\doteq  |e_0 \rangle \langle e_0 |+\sum_{i=0}^{k-1}S_2^iS_1S_1^*(S_2^*)^i$, which act on $\ell_2(\mathbb{Z})$.
It is then not difficult to see that now the $x_k$'s in the statement take on the much simpler form $x_k = (\alpha -1)|e_0 \rangle \langle e_0 |+Q_kd_kUd_k^*U^*$.
Now $Q_k$ strongly converges to the identity $I$, as shown by straightforward  computations. Furthermore, the sequence $Ud_k^*U^*$ converges too, since $d_k$ does.
This shows that $x_k$ 
is strongly convergent 
to a limit  $\delta\doteq (\alpha-1) |e_0 \rangle \langle e_0 | + d_\infty U d_\infty^* U^*$. Note that $\delta$ is a unitary lying in
$\ell_\infty(\mathbb{Z})=\D_2''$, for the equality see \cite[Sect. 2]{ACR}.
We also point out the
equality $d_k=d\varphi(d_{k-1})$, $k\in \{2, 3, \ldots \} \cup \{\infty\}$, 
which is necessary to carry out some of the following computations.\\
We can now deal with 4. We start with the if part. In view of Lemma \ref{unirel},
it is enough to show that $\tilde{S_2}\tilde{U}\doteq\tilde{U}^2\tilde{S_2}$ and $\tilde{S_1}\doteq\tilde{U}\tilde{S_2}$, where $\tilde{U}=\delta U \in \Q_2$ and $\tilde{S_i}=dS_i$.
As for the first equality, we rewrite $\delta U$ as
$$
\delta U= (\alpha-1) |e_0 \rangle \langle e_{-1} |+ d_\infty U d_\infty^*
$$
where 
$|e_0 \rangle \langle e_{-1} |(v)
= v(-1) e_0 $ for all $v\in\ell_2(\mathbb{Z})$.
If we now use the expression obtained above, we can compute $\tilde{U}^2\tilde{S_2}$ as 
$$
\delta U \delta U d S_2 = \Big((\alpha-1) |e_0 \rangle \langle e_{-1} |+ d_\infty U d_\infty^*\Big)^2 dS_2 = (\alpha-1) |e_0 \rangle \langle e_{-1} | + d_\infty U^2 d_\infty^* d S_2
$$
while
$$
dS_2 \delta U = dS_2 \Big((\alpha-1) |e_0 \rangle \langle e_{-1} |+ d_\infty U d_\infty^* \Big) = (\alpha-1) |e_0 \rangle \langle e_{-1} | + dS_2 d_\infty U d_\infty^* \ .
$$
It remains to show that $d_\infty U^2 d_\infty^* d S_2 = dS_2 d_\infty U d_\infty^*$. However, these two terms are equal to $d_\infty U^2 S_2 d_\infty^*$ and $d_\infty S_2 U d_\infty^*$, respectively,
where we used equation (\ref{weaklyinn}).
The second equality is dealt with by means of a still easier computation: 
$$\delta U d S_2 =\Big((\alpha-1) |e_0 \rangle \langle e_{-1} |+ d_\infty U d_\infty^*\Big)dS_2 = d_\infty U d_\infty^* d S_2 = d_\infty S_1 d_\infty^* = d S_1\ . $$
It is now clear that $\check{d} = (\alpha-1) |e_0 \rangle \langle e_0 | + d_\infty U d_\infty^* U^* \in \Q_2$ and it follows at once that $\check{d}(0)=\alpha$.\\
For the only if part, if we think of $\lambda_d$ as a representation of $\O_2$ on $\ell_2(\mathbb{Z})$, the conclusion follows by the uniqueness pointed out in \cite[Remark 4.2]{LarsenLi} (applied to $\check{d} U$).\\
Finally, the condition $\check{d}(1)=1$ leads to the equality $\check{d}=d_\infty U d_\infty^*U^*$, whence $\widetilde{\lambda_d}(U)=\check{d}U=d_\infty U d_\infty^*=\Ad(d_\infty)(U)$, which
says $\widetilde{\lambda_d}$ is implemented by $d_\infty$.
\end{proof}

Needless to say, if $\lambda_d$ is inner then (it extends to $\Q_2$ and) its extension in still inner. Conversely, if $\lambda_d$ extends to an inner automorphism of $\Q_2$, 
then $\lambda_d$ itself must be inner by the maximality of $\D_2$ in $\Q_2$ \cite[Sect. 3.1]{ACR}. Note that if $\lambda_d = \Ad (d')$, with $d' \in \U(\D_2)$ then 
$\check{d}(0) = d'(0) \overline{d'(-1)}$. If moreover $\check{d}(0) = 1$, then $d'$ is nothing but $d_\infty$, which is thus in $\D_2$.\\

Finally, it is not difficult to realize that the above theorem could also be set in other representations. Nevertheless, we shall refrain from discussing this issue any further, not least
because it goes beyond the scope of the present work.

\bigskip
For the results in the next section we need to take a closer look at inner diagonal automorphisms, to which the rest of this section is  addressed. To do that, we first prove
a general lemma.

 \begin{lemma}
For any $d\in\U(\D_2)$ the sequence $(S_2^*)^ndS_2^n$ converge normwise to the scalar $d(0)1$.
\end{lemma}
\begin{proof}
Given any $\varepsilon>0$ we can pick an algebraic $d_k\in\D_2^k$ such that $\|d-d_k\|<\frac{\varepsilon}{2}$.
As $\|(S_2^*)^ndS_2^n-d(0)1\|\leq \|(S_2^*)^n(d-d_k)S_2^n\|+\|(S_2^*)^nd_kS_2^n-d(0)1\|\leq\frac{\varepsilon}{2}+\|(S_2^*)^nd_kS_2^n-d(0)1\|$.
So the conclusion is obtained if we can also prove that $\|(S_2^*)^nd_kS_2^n-d(0)1\|$ may be made as small as wished.
To this aim, it is enough to write $d_k$ as a finite linear combination of projections, say $d_k=\sum_{|\alpha|=k}d_{\alpha}P_\alpha$, where all the coefficients $d_\alpha$ are in $\mathbb{T}$. Thought of as a  continuous function on the Cantor
set, the unitary $d_k$ is nothing but $d_k(x)=\sum_{|\alpha|=k}d_\alpha\chi_\alpha(x)$, and so there holds the inequality $\sup_{x\in\{1,2\}^\mathbb{N}}|d(x)-\sum_{|\alpha|=k}d_\alpha\chi_\alpha(x)|<\frac{\varepsilon}{2}$.
In particular, we also have $|d(\{222 \ldots\})-d_{22\ldots2}|=|d(0)-d_{22\ldots2}|<\frac{\varepsilon}{2}$. But as soon as $n$ is grater than $k$ the product $(S_2^*)^nd_kS_2^n$ reduces to $d_{22\ldots2}1$, and the conclusion is thus proved,
being more exactly $\|(S_2^*)^ndS_2^n-d(0)1\|\leq\varepsilon$.
\end{proof}

\begin{remark}
The same conclusion is also got to in the canonical representation. Indeed, since $(S_2^*)^ndS_2^n e_k= d(k2^n)e_{k}$ for any $k\in\mathbb{Z}$, the thesis amounts to
proving that $\lim_n d(k2^n)=d(0)$ for any $d\in\D_2$, with the limit being uniform in $k$. This is again easily proved by approximating in norm any such $d$ with algebraic unitaries as closely as necessary, which is much
the same as in the proof above.  
\end{remark}
In our next result there appear infinite products in a $C^*$-algebraic framework. Since this is a topic seldom discussed 
 in the literature, some comments 
as to which sense the convergence is understood in are necessary to dispel any possible doubt. The most refined notion one could work with essentially dates back
to von Neumann, and is the one intended with respect to the direct net of finite subsets of $\mathbb{N}$ ordered by inclusion. However, we shall not need to be that demanding.
For our purposes, we may as well make do with the usual notion that an infinite product $\prod a_i$ converges (in norm) if the sequence $p_n\doteq \prod_{i=1}^n a_i$ does.

\begin{proposition}\label{innerness}
Let $d$ be a unitary in $\D_2$ such that $d(-1)=d(0)=1$. Consider the following claims:
\begin{enumerate}
\item The automorphism $\lambda_d$ is inner.
\item The automorphism $\lambda_d$ extends to $\Q_2$ and the infinite product $\prod_{i=1}^\infty (S_2^*)^idS_2^i$ converges in norm to a unitary in $\D_2$.
\item There exists a unitary $d'\in\D_2$ such that $\lambda_d(S_2)=d'S_2d'^*$.
\end{enumerate}
Then we have the chain of implications $1\Rightarrow 2\Leftrightarrow 3$. Moreover, when $\lambda_d$ is inner the above  infinite product converges to $d'^*$, with
$d'$ being the unique unitary in $\D_2$ such that $d'(0)=1$ and $d'\varphi(d')^*=d$ as well as satisfying $\lambda_d(S_2)=d'S_2d'^*$.

\end{proposition}
\begin{proof}
For the implication $1\Rightarrow2$ we only need to prove that the infinite product converges. To this aim, let $d'\in\U(\D_2)$ such that
$d'(0)=1$ and $d'\varphi(d')^*=d$. By means of a simple computation by induction we gain the formula  $\prod_{i=1}^n (S_2^*)^idS_2^i=(S_2^*)^nd'S_2^n d'^*$. So the conclusion is arrived at by
a  straightforward application of the former lemma. For $2\Rightarrow 3$, we begin observing that it can be easily seen that $\prod_{j=1}^n (S_2^*)^jdS_2^j=(S_2^*)^n d_n S_2^n$.
By the following computations we get 
\begin{align*}
& \left(\prod_{j=1}^n (S_2^*)^jd^*S_2^j\right)S_2\left(\prod_{j=1}^n (S_2^*)^jdS_2^j\right) = \big((S_2^*)^n d_n^* S_2^n\big)S_2 \big((S_2^*)^n d_n S_2^n\big)\\
& = (S_2^*)^n d_n^* S_2 \big[S_2^n(S_2^*)^n\big] d_n S_2^n  \\
& = (S_2^*)^n d_n^* S_2 d_n S_2^n  \\
& = (S_2^*)^n d_n^*  \varphi(d_n) S_2 S_2^n  \\
& = (S_2^*)^n d^*  \varphi^n(d) S_2^{n+1}  \\
& = \big((S_2^*)^n d^*  S_2^n\big) d S_2    \; .
\end{align*}
Therefore, it follows that
\begin{align*}
& \lim_n \left(\prod_{j=1}^n (S_2^*)^jd^*S_2^j\right)S_2\left( \prod_{j=1}^n (S_2^*)^jdS_2^j\right)= \lim_n d (S_2^*)^n d^* S_2^n S_2\\
& = d \left( \lim_n (S_2^*)^nd^*S_2^n \right) S_2\\
& = d S_2 \; .
\end{align*}
So if we denote by $d'$ the limit of $\prod_{j=1}^n (S_2^*)^jd^*S_2^j$, we have the equality $d'S_2d'^*=dS_2$. Finally the implication $3\Rightarrow 2$  is dealt with by similar computations as in the proof of  $1\Rightarrow 2$. 
\end{proof}
\noindent
A well-known sufficient condition for the product in point 2 to exist (even with respect to the von Neumann notion of convergence)  is that $\sum_{i=1}^\infty \|1-(S_2^*)^id S_2^i\|<\infty$.

\section{Localized extendible automorphisms}
Whilst chiefly devoted to presenting our solution to the extension problem for localized automorphisms, this section also includes some results on general, i.e. possibly non-localized, unitaries.   
In fact, findings of this sort are interesting in their own right as well as being necessary tools to attack the case of localized automorphisms.  
First we prove a structural result about the diagonal elements with the extension property. 
\begin{theorem}\label{formula-d-check-d}
If $d\in \U(\D_2)$ is an extendible unitary, then $d$, $\check{d}$ and $U$ satisfy the following identity:
\begin{align}
d & =UdU^*\check{d}(S_1S_1^*+\varphi(\check{d})^*S_2S_2^*)\; . \label{formula-check}
\end{align}
\end{theorem}
\begin{proof}
The claim follows at once from Formulas \eqref{check-d-S1} and \eqref{check-d-S2}. Indeed,
\begin{align*}
& dS_1= \check{d}UdU^*S_1 \\
&  dS_2= \check{d}\varphi(\check{d}^*)  UdU^*S_2\; .
\end{align*}
and we are done.
\end{proof}
Now a handful of corollaries can be easily derived.
\begin{corollary}\label{dcheck}
Let $\lambda_d$, $d\in \U(\D_2)$, be an extendible automorphism. Then
$$
\check{d}= \check{d}(0) \prod_{i= 1}^\infty (S_2^*)^i d^* UdU^* S_2^i\; .
$$
\end{corollary}
\begin{proof}
By  formula \eqref{formula-check} 
we have that 
\begin{align*}
& \prod_{i=1}^k (S_2^*)^i dUd^*U^* S_2^i = \prod_{i=1}^k (S_2^*)^i \left( \check{d}(S_1S_1^*+\varphi(\check{d})^*S_2S_2^*) \right) S_2^i  \\
& = \prod_{i=1}^k (S_2^*)^i \left( \check{d}\varphi(\check{d})^*S_2S_2^*) \right) S_2^i = \prod_{i=1}^k (S_2^*)^i \left( \check{d}\varphi(\check{d})^*\right) S_2^i  \\
& = (S_2^*)^k (\check{d}\varphi(\check{d})^*)_k S_2^k =  (S_2^*)^k \check{d} \varphi^k(\check{d})^*S_2^k = (S_2^*)^k \check{d} S_2^k \check{d}^*
\end{align*}
which converges in norm to $\check{d}^*\check{d}(0)$ as $k\to +\infty$ by Proposition \ref{innerness}.
\end{proof}

\begin{corollary}
Let $\lambda_d$, $d\in \U(\D_2)$, be an extendible automorphism. Then the product $\prod_{i=1}^\infty (S_2^*)^i d S_2^i$ sits in $\D_2$ if and only if $\prod_{i=1}^\infty (S_2^*)^i UdU^* S_2^i$ does.
\end{corollary}
\begin{proof}
By formula \eqref{formula-check} we have that 
\begin{align*}
& \prod_{i=1}^k (S_2^*)^i d S_2^i = \prod_{i=1}^k (S_2^*)^i \left[ \check{d} UdU^*(S_1S_1^*+\varphi(\check{d})^*S_2S_2^*) \right] S_2^i  \\
& = \prod_{i=1}^k (S_2^*)^i \left[ \check{d}\varphi(\check{d})^* UdU^* S_2S_2^*) \right] S_2^i = \prod_{i=1}^k (S_2^*)^i \left[ \check{d}\varphi(\check{d})^* UdU^* \right] S_2^i  \\
& = \left(\prod_{i=1}^k (S_2^*)^i \check{d}\varphi(\check{d})^*S_2^i \right) \left(\prod_{i=1}^k (S_2^*)^i UdU^* S_2^i \right)\; .
\end{align*}
Since the sequence 
$$
\prod_{i=1}^k (S_2^*)^i \check{d}\varphi(\check{d})^*S_2^i 
$$
is norm-convergent as $k\to +\infty$ we have the claim.
\end{proof}
The result above does not depend on whichever topology  the convergence of the infinite products in the statement is understood in.\\

The former corollary applies in particular to localized unitaries. But in this case only finitely many terms occur in the infinite product that gives the corresponding $\check{d}$. More precisely, if $d\in\U(\D_2^k)$  then
 $\check{d} = \check{d}(0) \prod_{i = 1}^{k-1} (S_2^*)^i d^* UdU^* S_2^i\;$. As a result, the following corollary is proved forthwith.
\begin{corollary}\label{loccheck}
Let $k$ be a nonnegative integer
and let $d \in \U(\D_2^k)$ be such that $\lambda_d$ is extendible. Then $\check{d} \in \U(\D_2^{k-1})$.
\end{corollary}
This simple yet effective remark enables us to prove a preliminary result.
\begin{proposition}\label{loctrivial}
Let $b \in \U(\D_2^k)$ be such that $\lambda_b$ is extendible and $b S_2 = S_2$. Then $b=1$.
\end{proposition}
\begin{proof}
We proceed by induction. 
The case $k=1$ follows immediately from Proposition \ref{pointspectrum}.
Now let us suppose that $b \in \U(\D_2^k)$ with $k >1$.
Notice that one must have $b = b S_1 S_1^* + S_2 S_2^*$ and, moreover, $b S_1 = \lambda_b(S_1) = \lambda_b(US_2) = \check{b} U S_2 = \check{b} S_1$,
thus giving
$$b = \check{b} S_1S_1^* + S_2 S_2^* \ . $$
By Corollary \ref{loccheck}, the unitary $b$ must actually be in $\U(\D_2^{k-1})$ and we are done.
\end{proof}

We are now in a position to prove the main result of the section.

\begin{theorem}
Let $d \in \U(\D_2^k)$ be such that $\lambda_d$ is extendible. Then $\lambda_d$ is the product of a gauge automorphism and an inner localized automorphism.
\end{theorem}
\begin{proof}
As explained, we can suppose without loss of generality that $d(0) = d(-1) = 1$.
Proposition \ref{innerness} yields a $d'$ in $\U(\D_2^{k-1})$ such that ${\rm Ad}(d'^*) \circ \lambda_d(S_2) = S_2$.
The conclusion readily follows by Proposition \ref{loctrivial} applied to $b = d'^* \varphi(d') d \in \U(\D_2^k)$.
\end{proof}

At this point, it is worthwhile to stress that the proof of the main theorem also shows that for localized automorphisms the three conditions in Proposition \ref{innerness} are actually equivalent.
Indeed, whenever $d$ is localized the infinite product in \emph{2.} does converge being a finite product, and thus the unitary in \emph{3.} is localized too.    
There is yet another consequence of Proposition \ref{loctrivial} that deserves to be highlighted as a separate statement.
\begin{corollary}
The groups $\Aut_{\D_2}(\Q_2)_{\textrm{loc}}$ and $\Aut_{C^*(S_2)}(\Q_2)$ have trivial intersection.
\end{corollary}
One might wonder whether the subscript \emph{loc} may be got rid of altogether in the last statement.

\begin{remark}
The above discussion implies that if $d_n, d \in \U(\D_2)$ are such that $\lambda_{d_n}$ and $\lambda_d$ are extendible to $\Q_2$ and $\|d_n - d\| \to 0$ as $n \to \infty$ then
one has the continuity property 
$$\|\check{d_n} (S_1 S_1^* + \varphi(\check{d_n})^* S_2 S_2^*) - \check{d} (S_1 S_1^* + \varphi(\check{d})^* S_2 S_2^*)\| \to 0$$ 
as $n \to \infty$. If we multiply by $S_1S_1^*$ we see that $\|\check{d}_nS_1-\check{d}S_1\|$ goes to zero as well. In order to prove the continuity of the map $d\rightarrow \check{d}$ with respect to the norm topology, we should also check
$\lim_n \|\check{d}_nS_2-\check{d}S_2\|=0$. As a matter of fact, all we can prove is  the convergence only holds strongly. This is indeed a consequence of a general fact applied to $\{\check{d}^*\check{d_n}\}$: for any sequence 
$\{u_n\}\subset\U(\D_2)$ such that both
$\lim_n\|u_nS_1-S_1\|$ and $\lim_n\|u_nS_2u_n^*-S_2\|$ are zero, one also has $\lim_n |u_{n}(2k)-1|=0$, for every $k\in\mathbb{Z}$. However, the limit will in general fail to be uniform in $k$. This is
because we can only say $|u_{n}(2k)-1|$ is less than the sum $|u_{n}(2k)-u_n(k)|+|u_n(k)-u_n(\frac{k}{2})|+\ldots+|u_n\left(\frac{k}{2^h}\right)-1|$, with $h$ being $\doteq\min\{l: \frac{k}{2^l}\,\textrm{is odd}\}$, and so the number
of terms occurring in the sum depends upon what $k$ is (at worst, when $k=2^l$, exactly $l$ terms are needed).

\end{remark}

\section{Localizing $\check{d}$}
There is a last property of the group  homomorphism $\;\check{}\;$ that is well worth a thorough discussion. We have already observed that $\check{d}$ is localized whenever $d$ is.
Moreover, its kernel is obviously made up of localized unitaries. This seems to indicate that the other way around might also hold true, namely that an extendible $d\in\U(\D_2)$ whose $\check{d}$  is localized itself. 
This is quite the situation we are in. Far from immediate, the proof requires some preliminary work to do instead. More precisely, we need to give a closer look at the image of our homomorphism. 
To do so we point out that between  the entries of $d$ and $\widecheck{d}$ with respect to the canonical basis $\{e_k\}_{k\in\mathbb{Z}}$ in the canonical representation on $\ell_2(\mathbb{Z})$ there hold some relations.
More precisely, by evaluating  \eqref{Eq-ext-1} and \eqref{Eq-ext-2} on $e_{k-1}$ and $e_k$ respectively, we get the following couple of equations
\begin{align}
&   d(2k-1)   \check{d}(2k)=\check{d}(k) d(2k)	\label{eq-1}\\
& \check{d}(2k+1)d(2k) = d(2k+1)
\end{align}
to be satisfied by every integer $k$.
Furthermore, the former equation  can be recast  in terms of the canonical endomorphism $\varphi$ as 
$\widecheck{d}^*\varphi(\widecheck{d})S_2S_2^*=d^*UdU^*S_2S_2^*$.
Here below follows a technical result. 
\begin{lemma}
There exists a unique $d\in \ell_\infty(\mathbb{Z})$ such that $d(0)=1$ and $\check{d}=z\in \mathbb{T}$. In addition, it is given by $d=U_z\varphi(U_z)^*$ with $U_ze_k=z^k e_k, k \in \mathbb{Z}$.
\end{lemma}
\begin{proof} 
Substituting $\check{d}(k)=z$ in the former conditions, we find the equations
\begin{align*}
z d(2k) & = d(2k+1) \\
d(2k) & = d(2k-1) 
\end{align*}
whose solution is uniquely determined by $d(0)$. To conclude, just note that $U_z\varphi(U_z)^*$ satisfies the equations thanks to $\varphi(d)(k)=d([k/2])$, $k\in\mathbb{Z}$.
\end{proof}
\begin{proposition}\label{prop-intersec-tor-check}
The image of $\;\check{}\,$ intersects $\mathbb{T}$ on the set of all roots of unity of order $2^k$, $k\in\mathbb{N}$.
\end{proposition}
\begin{proof}
Whenever $z$  is a root of unity of order a power  $2^k$ the corresponding $U_z$ is in $\D_2$, as shown in \cite{ACR}, and $\check{d}=z$ with $d\doteq U_z\varphi(U_z)^*$.
To conclude, we need to prove that nothing else is contained in the intersection. To this aim, note that it is not restrictive to assume
$d(0)=1$ if $d$ satisfies $\check{d}=z$. This allows us to consider $d$ of the form $U_z\varphi(U_z)^*$. There are two cases to deal with. They can be both worked out with some ideas
borrowed from \cite{ACR}.   
In the first, $z$  is a root of unity, but not of order a power  of $2$.
This is handled as follows. Since any projection $P\in \D_2$ is in the linear algebraic span of $\{S_\alpha S_\alpha^*\}_{\alpha \in W}$, we have that $U_z\varphi(U_z)^*=\sum_{|\alpha|=h}c_\alpha S_\alpha S_\alpha^*$. Without loss of generality, we may suppose that $h\geq 2$. By definition, $U_z\varphi(U_z)^* e_0=e_0$ so that $c_{(2,\ldots , 2)}=1$. The equality $U_z\varphi(U_z)^* e_{2^h l}=z^{2^{h-1} l}e_{2^h l}$ gives $z^{2^{h-1} l}=1$, which
is absurd as soon as $l=1$.\\
In the second case,  $z$ is not a root of unity. If $U_z\varphi(U_z)^*$ did belong to $\D_2$, then there would exist 
a positive integer $k$ such that
$$\big\|U_z\varphi(U_z)^*-\sum_{|\alpha |=k} c_\alpha S_\alpha S_\alpha^*\big\|<\varepsilon$$
with $\sum_{|\alpha |=k} c_\alpha S_\alpha S_\alpha^*$ being unitary as well.
In particular, we would have 
$$\| z^i e_{2i}-\sum_{|\alpha |=k} c_\alpha S_\alpha S_\alpha^*e_{2i}\|<\varepsilon$$ for every integer $i$.
Taking $i=0$ we would thus get
$$
\| e_0-\sum_{|\alpha |=k} c_\alpha S_\alpha S_\alpha^*e_0\|=|1- c_{(2,\ldots, 2)}| <\varepsilon\; ,
$$
and, in particular, for $i=2^{k-1}l$
$$
\| z^{2^{k-1}l}e_{2^kl}-\sum_{|\alpha |=k} c_\alpha S_\alpha S_\alpha^*e_{2^kl}\|=|z^{2^{k-1}l}- c_{(2,\ldots, 2)}| <\varepsilon\; 
$$
As $z$ is not a root of unity, then $z^{2^{k-1}}$ is not a root of unity either. Therefore, $\{ (z^{2^{k-1}})^l \}_{l\in \mathbb{Z}}$ is dense in $\mathbb{T}$, which contradicts the inequality above. 

\end{proof}

Furthermore, there is a simple invariance property we would like to stress.
\begin{lemma}\label{invariance-Ad-U}
The image of  $\;\; \check{ }\; $ is invariant under $\Ad(U^k)$, for any $k\in\mathbb{Z}$.
\end{lemma}
\begin{proof}
Set $\alpha= \Ad(U^k)\circ\lambda_d\circ\Ad(U^{-k})$. Then the equalities $\alpha(U)=U^k\check{d}U^{-k}U$ and $\alpha(x)= x$ for any $x\in \D_2$ are both easily verified. The conclusion is thus arrived at.
\end{proof}
We have finally gathered all the tools necessary to carry out our proof.
\begin{theorem}
Let $d$ be an extendible unitary in $\D_2$, then $d$ is localized if and only if $\check{d}$ is localized.
\end{theorem}
\begin{proof}
Obviously, all we have to do is prove that $d$ is algebraic if $\check{d}$ is. If $\check{d}$ is in $\D_2^k$, then it takes the form
$\check{d}=\sum_{i=0}^{2^k-1}c_i U^iS_2^k(S_2^k)^*U^{-i}$, where the $c_i$'s are all in $\mathbb{T}$. By the above lemma the product 
$\prod_{i=0}^{2^k-1} U^i\check{d}U^{-i}=\big(\prod_{i=0}^{2^k-1} c_i \big)1$ is still in the range of the map $\;\check{ }\;$. Now
Proposition \ref{prop-intersec-tor-check} tells us that $\prod_{i=0}^{2^k-1} c_i$ is a root of unity of order $2^k$, for some $k\in\mathbb{N}$. \\
Suppose at first that $\prod_{i=0}^{2^k-1} c_i=1$.
We are therefore led to verify that there exists a $\tilde{d}\in\D_2^k$ such that $\tilde{d}U\tilde{d}^*U^*=\check{d}$.
As above, we can write $\tilde{d}$ as $\sum_{i=0}^{2^k-1}\alpha_i U^iS_2^k(S_2^k)^*U^{-i}$, which allows to recast the equation for $\tilde{d}$ in terms of its coefficients $\alpha_i$
\begin{align*}
 c_0  & = \alpha_0 \overline{\alpha_{2^k-1}}\\
 c_1  & = \alpha_1 \overline{\alpha_{0}}\\
 \vdots & \\
 c_{2^k-1} & = \alpha_{2^k-1} \overline{\alpha_{2^k-2}}
\end{align*}
To begin with, it is not restrictive to suppose $\alpha_0=1$. With this choice the solution is given by $\alpha_i= c_i c_{i-1}\cdots c_1$ for every $i\geq 1$. Note that the condition $\prod_{i=0}^{2^k-1} c_i =1$ guarantees the compatibility of the system.  Now we can deal with the general case, i.e. $\prod_{i=0}^{2^k-1} c_i=e^{2\pi i h/2^n}$ for some $h, n\in\mathbb{N}\setminus \{0\}$.
If we consider $\check{d'}\doteq \check{d}e^{-2\pi i h/2^{n+k}}$, then Proposition \ref{prop-intersec-tor-check} implies 
that $\check{d'}\in \{dUd^*U^*: d\in\U(\bigcup_k \D_2^k)\}$ if and only if $\check{d}\in \{dUd^*U^*: d\in\U(\bigcup_k \D_2^k)\}$. But now we have that
$$
\prod_{i=0}^{2^k-1} U^i\check{d'}U^{-i}=e^{-2\pi i h/2^{n}} \left(\prod_{i=0}^{2^k-1} c_i \right) =e^{-2\pi i h/2^{n}}e^{2\pi i h/2^{n}}=1
$$
and the thesis follows from the fist part of the argument.
\end{proof}

%


\section{An outlook for further results}

A natural way to complete the analysis hitherto conducted would 
require to attempt a generalization of Proposition \ref{loctrivial} to extend its reach to every $d\in\U(\D_2)$. 
In this regard, it is worth stressing that if $d \in \U(\D_2)$ is an extendible unitary satisfying $dS_2 = S_2$ then $\check{d}$ must satisfy
$$
\check{d}=\check{d}(0)\prod_{i= 1}^\infty (S_2^*)^i UdU^* S_2^i
$$ 
along with the couple of equations
$$
\left\{ \begin{array}{l}
\check{d} P_1 = d P_1\\
\check{d}P_{2_r 1}=(Ud^*U^*)_r\varphi^r(d)P_{2_r 1}, \, \textrm{for every}\,r\geq 1 \\
\end{array} \right.
$$\\
as follows from Theorem \ref{formula-d-check-d} and an easy induction. Here, we used $2_r$ for $2 \cdots 2$ ($r$ times) to ease the notation.
Surprisingly enough, even though $\check{d}$ is (over)determined by both the above formulas,
the sought extension of Proposition \ref{loctrivial} has nonetheless proved to be a difficult goal to accomplish.
A class of ostensively easier problems, however, arises as soon as further assumptions are made. One could for instance introduce the extra condition $d S_1^k = S_1^k$, for some $k \geq 2$, as well
as $dS_2=S_2$. Still, the problem thus obtained is a hard game to play.  Indeed, even dealing with the easiest case, i.e. $k=2$, requires a close look at the 
 odometer, see for instance \cite{kerr}, 
 \cite[p. 230, 231]{Davidson}. This is the map $T\in\textrm{Homeo}(K)$ given by
 
\begin{align*}
(T\{x_i\}_{i\geq 1})_j & 
\doteq\left\{ \begin{array}{cc}
1 & j<k\\
2 & j=k\\
x_j & j>k\\
\end{array} \right.
\end{align*}
with $k\doteq\inf\{j\; | \; x_j= 1\}\in \mathbb{N}\cup \{+\infty\}$, and $+\infty$ is attained only on the constant sequence $\{222\ldots\}$. As recalled in Section 2, the map $T$ implements the restriction of $\Ad(U)$
to the diagonal $\D_2$, that is $UdU^*(x)=d(Tx)$ for every $x\in\{1,2\}^\mathbb{N}\cong K$.\\
%

Here is the precise statement of the result alluded to above.
\begin{proposition}
The only extendible automorphism $\lambda_d$, $d \in \U(\D_2)$,  such that $dS_2=S_2$ and $dS_1^2=S_1^2$ is the identity.
\end{proposition}
\begin{proof}
Under our hypotheses, Formula  \eqref{Eq-ext-2} at the beginning of Section 4 becomes $\check{d}S_1=dS_1$, which at the spectrum level reads
$d(1x)=\check{d}(1x)$ for all $x\in\{1,2\}^\mathbb{N}$. Formula  \eqref{Eq-ext-1} leads to $S_2^* \check{d} U\check{d}U^*S_2 =\check{d}$, which turns into $\check{d}(2 x)\check{d}(1T(x))=\check{d}(x)$
thanks to the equality $(S_2^*dS_2)(x)=d(2x)$. Now the condition $dS_1^2=S_1^2$ gives $\check{d}(11x)=d(11x)=1$ for all $x$.
Picking $x=2_i1 y$ we get that $\check{d}(2_{i+1}1y)\check{d}(1_{i+1}2y)=\check{d}(2_i1y)$ for all $i\geq 1$, and thus $\check{d}(2_{i+1}1y)=\check{d}(2_i1y)=\check{d}(21y)$. This says that the sequence  
$\{\check{d}(2_{i+1}1y)\}$ is constant. Therefore, each term of it equals its limit, i.e. $\check{d}(22\ldots)$. Now $\check{d}(2  2\ldots)=\check{d}(0)$  (in the canonical representation), 
 hence 
$\check{d}(2_{i}1y)=\check{d}(0)\doteq \alpha$ for every $y\in\{1,2\}^{\mathbb{N}}$. In particular, we find that $\check{d}(2y)=\alpha$ for all $y$. 
Under our hypothesis (and by the previous computations) we know that in the canonical representation $d(2k+1)=\check{d}(2k+1)$, $d(2k)=1$, and $\check{d}(2k)=\alpha$ for all $k\in \mathbb{Z}$. Thus if we plug  $k=1$ in \eqref{eq-1} we get 
$$
d(1)\alpha= d(1) \check{d}(2) = \check{d}(1) d(2)= d(1)d(2)=d(1)
$$
which shows that $\alpha =1$.
The conclusion is now achieved by a straightforward application of Proposition \ref{fixed-point-check}.
\end{proof}

\end{document}